\documentclass[11pt]{article}
\usepackage{amsthm,amssymb,amsmath}
\usepackage{epsfig}
\usepackage{verbatim}

\title{Planar graphs have exponentially many 3-arboricities}

\author{Ararat Harutyunyan\thanks{Research supported by FQRNT
  (Le Fonds qu\'{e}b\'{e}cois de la recherche sur la nature et les technologies)
  doctoral scholarship.}\\
  {Department of Mathematics}\\
  {Simon Fraser University}\\
  {Burnaby, B.C. V5A 1S6} \\
  email: {\tt aha43@sfu.ca}
\and
  Bojan Mohar\thanks{Supported in part by an NSERC Discovery Grant (Canada),
  by the Canada Research Chair program, and by the
  Research Grant P1--0297 of ARRS (Slovenia).}~\thanks{On leave from:
  IMFM \& FMF, Department of Mathematics, University of Ljubljana, Ljubljana,
  Slovenia.}\\
  {Department of Mathematics}\\
  {Simon Fraser University}\\
  {Burnaby, B.C. V5A 1S6} \\
  email: {\tt mohar@sfu.ca}
}

\newtheorem{theorem}{Theorem}[section]
\newtheorem{lemma}[theorem]{Lemma}
\newtheorem{corollary}[theorem]{Corollary}
\newtheorem{proposition}[theorem]{Proposition}
\newtheorem{conjecture}[theorem]{Conjecture}

\newcommand{\DEF}[1]{{\em #1\/}}

\newcommand{\C}{{\cal C}}

\begin{document}

\maketitle

\begin{abstract}
It is well-known that every planar or projective planar graph can
be 3-colored so that each color class induces a forest. This bound
is sharp. In this paper, we show that there are in fact
exponentially many 3-colorings of this kind for any (projective)
planar graph. The same result holds in the setting of
3-list-colorings.
\end{abstract}


\section{Introduction}

Motivation for this paper comes from two directions. One is
related to the arboricity of undirected planar graphs, the other
one to colorings of planar digraphs. Let us recall that a
partition of vertices of a graph $G$ into classes $V_1 \cup \,
\cdots \cup V_k$ is an \DEF{arboreal partition} if each $V_i$ ($1
\leq i \leq k$) induces a forest in $G$. A function $f \colon V(G)
\to \{1,\dots,k\}$ is called an \DEF{arboreal $k$-coloring} if
$V_i = f^{-1}(i)$, $i=1,\dots,k$, form an arboreal partition. The
\DEF{vertex-arboricity} $a(G)$ of the graph $G$ is the minimum $k$
such that $G$ admits an arboreal $k$-coloring. Note that $a(G)
\leq \chi(G) \leq 2 a(G)$, where $\chi(G)$ is the chromatic number
of $G$. Long ago, people asked if every planar graph has
arboricity 2 since this would imply the Four Color Theorem.
However, planar graphs of vertex-arboricty 3 have been found (see
Chartrand et al. \cite{CKW1968}).

Let $D$ be a digraph without cycles of length $\leq 2$, and let
$G$ be the underlying undirected graph of $D$. A function $f
\colon V(D) \to \{1,\dots,k \}$ is a \DEF{k-coloring} of the
digraph $D$ if $V_i = f^{-1}(i)$ is acyclic in $D$ for every
$i=1,\dots,k$. Here we treat the vertex set $V_i$ \DEF{acyclic} if
the induced subdigraph $D[V_i]$ contains no directed cycles (but
$G[V_i]$ may contain cycles). The minimum $k$ for which $D$ admits
a $k$-coloring is called the \DEF{chromatic number} of $D$, and is
denoted by $\chi(D)$ (see Neumann-Lara \cite{N1982}). Clearly,
$$ \chi(D) \leq a(G).$$

While planar graphs with arboricity 3 are known, no planar digraph
(without cycles of length $\leq 2$) with $\chi(D) > 2$ is known.
In fact, the following conjecture was proposed independently by
Neumann-Lara \cite{N} and \v{S}krekovski in \cite{BFJKM2004}.

\begin{conjecture} \label{conj:planar}
Every planar digraph $D$ with no directed cycles of length at most
$2$ has $\chi(D) \leq 2$.
\end{conjecture}

It is an easy consequence of 5-degeneracy  of planar graphs that
every planar digraph $D$ without cycles of length at most 2 and
its associated underlying planar graph $G$ satisfy
\begin{equation} \label{eqn:1}
\chi(D) \leq a(G) \leq 3.
\end{equation}

The main result of this paper is a relaxation of Conjecture
\ref{conj:planar} and a strengthening of the above stated
inequality (\ref{eqn:1}). In doing so, we also extend the result
from planar graphs to graphs embedded in the projective plane. In
particular, we prove the following.

\begin{theorem} \label{thm:main} Every planar or projective planar
graph of order $n$ has at least\/ $2^{n/9}$ arboreal
$3$-colorings.
\end{theorem}

\begin{corollary} \label{cor:main}
Every planar or projective planar digraph of order $n$ without
cycles of length at most\/ $2$ has at least\/ $2^{n/9}\,$
$3$-colorings.
\end{corollary}

Let us observe that Theorem \ref{thm:main} cannot be extended to
graphs embedded in the torus since $a(K_7)=4$ and $K_7$ admits an
embedding in the torus. However, for every orientation $D$ of
$K_7$, we have $\chi(D) \leq 3$ (and in some cases $\chi(D)=3$);
and it follows from the main result in \cite{HM2011} that every
orientation of a (simple) graph embeddable in the torus satisfies
$\chi(D) \leq 3$. So it is possible that Corollary \ref{cor:main}
extends to the torus. Graphs on the Klein Bottle behave nicer
since $K_7$ can not be embedded in the Klein Bottle.
\v{S}krekovski \cite{S2002} and Kronk and Mitchem \cite{KM1974}
have shown that these graphs have arboricity at most 3.

It can be shown that a graph on the torus has arboricity at most 3
unless it contains $K_7$ as a subgraph. This can be used to prove
that for every graph $G$ embeddable in the torus, there exists an
edge $e \in E(G)$ such that $a(G - e) \leq 3$. In this vein, we
conjecture the following.

\begin{conjecture} For every graph $G$ embeddable in the torus,
there exists an edge $e \in E(G)$ such that $G - e$ has
exponentially many 3-arboreal colorings.
\end{conjecture}

The proof of Theorem \ref{thm:main} is deferred until Section 4.
Actually, we shall prove an extended version in the setting of
list-colorings which we define next.

Let $\C$ be a finite set of colors. Given a graph $G$, let $L:
v\mapsto L(v)\subseteq \C$ be a \DEF{list-assignment} for $G$,
which assigns to each vertex $v\in V(G)$ a set of colors. The set
$L(v)$ is called the \DEF{list} (or the set of \DEF{admissible
colors}) for $v$. We say $G$ is \DEF{$L$-colorable} if there is an
\DEF{$L$-coloring} of $G$, i.e., each vertex $v$ is assigned a
color from $L(v)$ such that every color class induces a forest in
$G$. A \DEF{$k$-list-assignment} for $G$ is a list-assignment $L$
such that $|L(v)|=k$ for every $v \in V(G)$.

\begin{theorem} \label{thm: main 3-list} Let $L$ be a $3$-list-assignment for a planar
or projective planar graph $G$ of order $n$. Then $G$ has at
least\/ $2^{n/9}$ $L$-colorings.
\end{theorem}

Similarly, we define list colorings for digraphs, where we insist
that color classes induce acyclic subdigraphs. Corollary
\ref{cor:main} then extends, as a corollary to Theorem \ref{thm:
main 3-list} to the list coloring setting as well.

\section{Unavoidable configurations}

We define a \DEF{configuration} as a plane graph $C$ together with
a function $\delta \colon V(C) \to \mathbb{N}$ such that
$\delta(v) \geq deg_C(v)$ for every $v \in V(C)$. A plane graph
$G$ \DEF{contains} the configuration $(C, \delta)$ if there is an
injective mapping $h \colon V(C) \to V(G)$ such that the following
statements hold:

\begin{enumerate}
\item[(i)] For every edge $ab \in E(C)$, $h(a)h(b)$ is an edge of
$G$.
\item[(ii)] For every facial walk $a_1 \dots a_k$ in $C$,
except for the unbounded face, the image $h(a_1) \dots h(a_k)$ is a
facial walk in $G$.
\item[(iii)] For every $a \in V(C)$, the degree of $h(a)$ in $G$
is equal to $\delta(a)$.
\end{enumerate}

If $v$ is a vertex of degree $k$ in $G$, then we call it a
\DEF{$k$-vertex}, and a vertex of degree at least $k$ (at most
$k$) will also be referred to as a \DEF{$k^{+}$-vertex}
(\DEF{$k^{-}$-vertex}). A neighbor of $v$ whose degree is $k$ is a
\DEF{$k$-neighbor} (similarly \DEF{$k^{+}$-} and
\DEF{$k^{-}$-neighbor}).

The goal of this section is to prove the following theorem.

\begin{theorem}
\label{thm:unavoidable}
Every planar or projective planar triangulation contains one of the
configurations listed in Fig.~\ref{fig:1}.
\end{theorem}

\begin{figure}[t!]
   \centering
   \includegraphics[width=11.2cm]{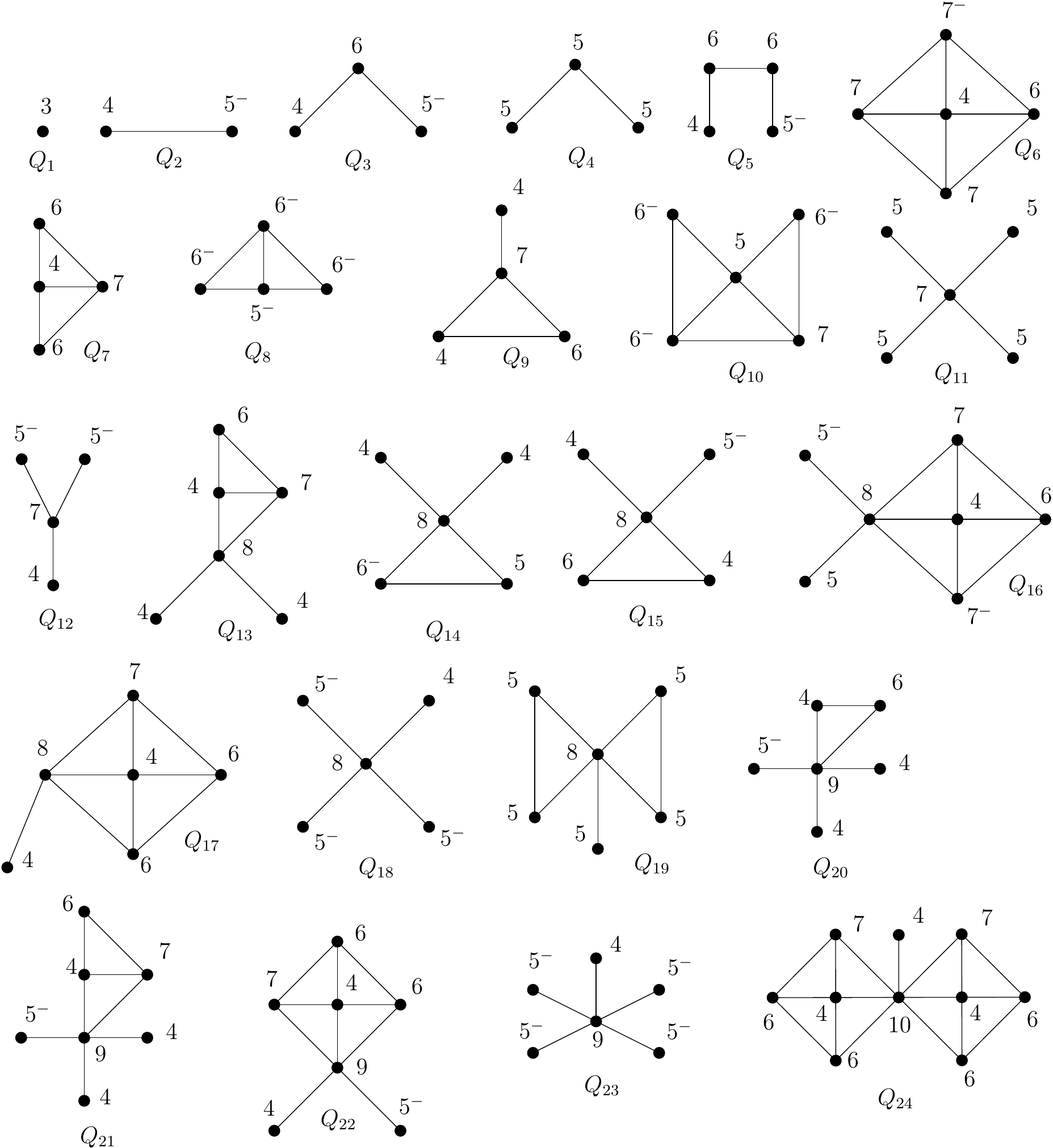}
   \caption{Unavoidable configurations. The listed numbers refer to the degree function $\delta$,
   and the notation $d^{-}$ at a vertex $v$ means all such configurations where the value $\delta(v)$ is
   either $d$ or $d-1$.}
   \label{fig:1}
\end{figure}

\begin{proof} The proof uses the discharging method. Assume, for a
contradiction, that there is a (projective) planar triangulation
$G$ that contains none of the configurations shown in
Fig.~\ref{fig:1}. We shall refer to these configurations as $Q_1,
Q_2, \dots, Q_{23}$.

Let $G$ be a counterexample of minimum order. To each vertex $v$
of $G$, we assign a {\em charge} of $c(v)=\deg(v) - 6$. A well-known
consequence of Euler's formula is that the total charge is always
negative, $\sum_{v \in V(G)} c(v) = -12$ in the plane and $\sum_{v
\in V(G)} c(v) = -6$ in the projective plane, see \cite{MT2001}.
We are going to apply the following \DEF{discharging rules}:

\begin{itemize}
\item[R1:] A 7-vertex sends charge of $1/3$ to each adjacent
5-vertex.

\item[R2:] A 7-vertex sends charge of $1/2$ to each adjacent
4-vertex.

\item[R3:] An $8^{+}$-vertex sends charge of $1/2$ to each
adjacent 5-vertex.

\item[R4:] An $8^{+}$-vertex sends charge of $2/3$ to each
adjacent 4-vertex whose neighbors have degrees $8^{+}$, $8^{+}$,
$8^{+}$, 6.

\item[R5:] An $8^{+}$-vertex sends charge of $3/4$ to each
adjacent 4-vertex whose neighbors have degrees $8^{+}$, $8^{+}$,
$7$, 6.

\item[R6:] An $8^{+}$-vertex sends charge of $1/2$ to each
adjacent 4-vertex whose neighbors have degrees $8^{+}, 7^{+},
7^{+}, 7^{+}$.

\item[R7:] An $8^{+}$-vertex sends charge of $1$ to each adjacent
4-vertex whose neighbors have degrees $8^{+},8^{+},6,6$ or
$8^{+},7,7,6$.

\item[R8:] An $8^{+}$-vertex sends charge of $3/2$ to each
adjacent 4-vertex whose neighbors have degrees $8^{+},7, 6, 6$.
\end{itemize}
Let $c^{*}(v)$ be the \DEF{final charge} obtained after applying
rules R1--R8 to all vertices in $G$. We will show that every
vertex has non-negative final charge. This will yield a
contradiction since the initial total charge of $-12$ (or $-6$ in
the projective plane) must be preserved.

We say that a $4$-vertex is \DEF{bad} if its neighbors have
degrees $8^{+}, 7, 6, 6$, i.e., the rule R8 applies to it and its
$8^{+}$-neighbor. Let us observe that the clockwise order of
degrees of the neighbors of a bad vertex is $8^{+},7,6,6$ (or
$8^{+},6,6,7$) since $Q_7$ is excluded.

First, note that $G$ has no $3^{-}$-vertices since the
configuration $Q_1$ is excluded and since a triangulation cannot
have $2^{-}$-vertices. We will also have in mind that $Q_2$ is
excluded, so every neighbor of a 4-vertex is a $6^{+}$-vertex.

\bigskip

\textbf{4-vertices:} Let $v$ be a 4-vertex. Note that $v$ has only
$6^{+}$-neighbor. If all neighbors have degree at most 7, then
they all have degree exactly 7 since $Q_6, Q_7$ and $Q_8$ are
excluded. Since the vertex $v$ has initial charge of $-2$, and
each 7-neighbor sends a charge of $1/2$ to it, the final charge of
$v$ is $0$.

Now, assume that $v$ is adjacent to an $8^{+}$-vertex. First,
assume that the remaining three neighbors $v_1, v_2, v_3$ of $v$
are all $7^{-}$-vertices. The vertices $v_1,v_2,v_3$ cannot all
have degree 6 since $Q_8$ is excluded. If $\deg(v_1) = 7$ and
$\deg(v_2) = \deg(v_3) = 6$, then the rules R2 and R8 imply that
$v$ receives a charge of $2$, resulting in the final charge of
$0$. If $\deg(v_1) = \deg(v_2) = 7$ and $\deg(v_3) =6$, then by
rules R2 and R7, $v$ again receives a charge of $2$. The case
where $\deg(v_1) = \deg(v_2) = \deg(v_3) = 7$ is similar through
rules R2 and R6.

Next, assume that $v$ has exactly two $8^{+}$-neighbors $v_1,
v_2$. If the remaining two vertices $v_3, v_4$ are both
7-vertices, then rules R2 and R6 imply that $v$ receives a total
charge of $2$, giving it the final charge of $0$. If the remaining
two vertices are both 6-vertices, then rule R7 implies that $v$
receives a total charge of $2$, resulting in $0$ final charge.
Therefore, we may assume that $\deg(v_3) = 7$ and $\deg(v_4) = 6$.
In this case, both $v_1$ and $v_2$ send a charge of $3/4$ to $v$
by R5, and $v_3$ sends a charge of $1/2$, resulting in a final
charge of $0$ for $v$.

Finally, assume that $v$ has at least three $8^{+}$-neighbors. By
rule R4 (if $v$ has a 6-neighbor), or by rules R2 and R6 (if $v$
has a 7-neighbor), or by rule R6 (otherwise), we see that $v$
receives total charge of $2$, so $c^{*}(v) = 0$.

\medskip
\textbf{5-vertices:} Let $v$ be a 5-vertex. Note that $v$ is not
adjacent to a 4-vertex. If all neighbors of $v$ are
$7^{-}$-vertices, then exclusion of $Q_4$, $Q_8$ and $Q_{10}$
implies that $v$ has at least three 7-neighbors. By R1, each such
neighbor sends a charge of $1/3$ to $v$. Since $v$ has initial
charge of $-1$, its final charge is at least $0$. Next, suppose
that $v$ has an $8^{+}$-neighbor. If $v$ has at least two
$8^{+}$-neighbors, then by rule R3, $v$ receives a charge of $1/2$
from each of them, yielding $c^{*}(v) \geq 0$. Therefore, we may
suppose that $v$ has exactly one $8^{+}$-neighbor. If $v$ has at
least two 7-neighbors, then by R1 and R3, $v$ receives a total
charge of at least $1/2 + 1/3 + 1/3 > 1$, resulting in a positive
final charge for $v$. Finally, if $v$ has at most one 7-neighbor,
then we get the configuration $Q_4$, $Q_8$ or $Q_{10}$.

\medskip
\textbf{6-vertices:} They have initial charge of $0$, and by the
discharging rules, they do not give or receive any charge, which
implies that they have a final charge of $0$.

\medskip
\textbf{7-vertices:} Let $v$ be a 7-vertex, and note that $v$ has
initial charge of 1. If $v$ has no 4-neighbors then it has at most
three 5-neighbors since $Q_{11}$ is excluded. Therefore, it sends
a charge of $1/3$ to each such vertex, resulting in a non-negative
final charge. Next, suppose that $v$ has at least one 4-neighbor.
Since $Q_{12}$ is excluded, $v$ has at most one other
$5^{-}$-neighbor. Therefore, $v$ sends a charge of at most $1/2 +
1/2 = 1$, resulting in the final charge of at least $0$ for $v$.

\medskip
\textbf{8-vertices:} An 8-vertex $v$ has initial charge of $+2$.
Since $Q_{17}$ is excluded, $v$ has at most three 4-neighbors.
First, suppose that $v$ has exactly three 4-neighbors. Let $u$ be
one of them. We claim that $v$ sends charge of at most $2/3$ to
$u$. Since $Q_{15}$ and $Q_{13}$ are excluded, we have that $N(u)
\backslash \{v\}$ contains vertices of degrees either $7^{+},
7^{+}, 7^{+}$ or $8^{+}, 8^{+}, 6$. In the first case, $v$ sends
charge $1/2$ to $u$, and in the second case charge $2/3$. Since
$v$ has no 5-neighbors (again, by exclusion of $Q_{17}$),
$c^{*}(v) \geq 2 - 3 \times 2/3 = 0$.

Next, suppose that $v$ has exactly two 4-neighbors, say $v_1$ and
$v_2$. We consider two subcases. First, assume that $v$ has a
5-neighbor. Excluding $Q_2$ and $Q_{14}$, no vertex in $N(v_1)
\cap N(v)$ and $N(v_2) \cap N(v)$ has degree at most 6. If the two
vertices in $N(v_1) \cap N(v)$ are both 7-vertices, then $v_1$ has
no $6^{-}$-neighbor ($Q_2$ and $Q_{15}$ being excluded). This
implies that $v$ sends charge of $1/2$ to $v_1$. Otherwise, the
two vertices in $N(v_1) \cap N(v)$ are an $8^{+}$ and a
$7^{+}$-vertex, respectively. This implies that by rules R4, R5
and R6, $v$ sends charge of $3/4$, $2/3$ or $1/2$ to $v_1$.
Therefore, in all cases, $v$ sends no more than $3/4$ charge to
$v_1$. An identical argument shows that $v$ sends a charge of at
most $3/4$ to $v_2$. Since $v$ sends a charge of $1/2$ to a
5-vertex, we have that $v$ sends a total charge of at most $3/4 +
3/4 + 1/2 = 2$. Secondly, assume that $v$ has no 5-neighbors.
Consider $v_1$. Excluding $Q_7$ and $Q_{16}$, $v_1$ is not a bad
4-vertex. Therefore, $v$ sends charge of at most $1$ to $v_1$. An
identical argument shows that $v$ sends charge of at most $1$ to
$v_2$. Therefore, the final charge of $v$ is non-negative.

Next, suppose that $v$ has exactly one 4-neighbor, say $v_1$.
First, suppose that $v_1$ is a bad 4-vertex. Excluding $Q_7$ and
$Q_{15}$, $v$ has at most one 5-neighbor. Since $v$ sends a charge
of at most $3/2$ to $v_1$ and charge $1/2$ to its 5-neighbor, its
final charge is at least $0$. Thus, we may assume that $v_1$ is
not a bad 4-vertex. Then $v$ sends at most charge of $1$ to $v_1$.
Because $Q_{17}$ is excluded, $v$ has at most two 5-neighbors, to
each of which it sends a charge of $1/2$. Therefore, $v$ sends a
total charge of at most $1 + 1/2 + 1/2= 2$, which implies that
$c^{*}(v) \geq 0$.

Finally, suppose that $v$ has no 4-neighbors. Excluding $Q_{18}$
and $Q_4$, $v$ has at most four 5-neighbors, to each of which it
sends charge of $1/2$. Therefore, the final charge of $v$ is
non-negative.

\medskip
\textbf{9-vertices:} A $9$-vertex $v$ has a charge of +3. Since
$Q_{22}$ is excluded, $v$ has at most four 4-neighbors. First,
suppose that $v$ has exactly four 4-neighbors or three 4-neighbors
and at least one 5-neighbor; let $u$ be one of the 4-neighbors. We
claim that $v$ sends charge of at most $2/3$ to $u$. Since
$Q_{20}$ and $Q_{19}$ are excluded, we have that $N(u) \backslash
\{v\}$ contains vertices of degrees $7^{+}, 7^{+}, 7^{+}$ or
$8^{+}, 8^{+}, 6$. In the first case, $v$ sends charge $1/2$ to
$u$, and in the second case charge $2/3$. Since $v$ has only one
5-neighbor (again, by exclusion of $Q_{22}$), $c^{*}(v) \geq 3 - 4
\times 2/3
> 0$.

Next, suppose that $v$ has exactly three 4-neighbors and no
5-neighbors. Since $Q_7$ and $Q_{21}$ are excluded, none of the
4-neighbors are bad. Therefore, in this case $v$ sends charge of
at most 1 to each 4-neighbor, resulting in a non-negative final
charge.

If $v$ has exactly two 4-neighbors, we consider two subcases. For
the first subcase, suppose that none of the 4-neighbors are bad.
Now, $v$ has at most two $5$-neighbors since $Q_{22}$ is excluded.
This implies that $v$ sends total charge of at most $1 + 1 + 1/2 +
1/2 = 3$ to its neighbors, resulting in a non-negative final
charge for $v$. For the second subcase, assume that $v$ has at
least one bad 4-neighbor. Now, the exclusion of $Q_{21}$ implies
that $v$ has no $5$-neighbors. Thus, $v$ sends total charge of at
most $3/2 + 3/2 = 3$, and therefore $c^{*}(v) \geq 0$.

Suppose now that $v$ has exactly one 4-neighbor. The exclusion of
$Q_{22}$ implies that $v$ has at most three 5-neighbors, and hence
it sends out a total charge of at most $3/2+ 1/2 + 1/2 + 1/2 = 3$,
resulting in $c^{*}(v) \geq 0$. Lastly, assume that $v$ has no
4-neighbors. Excluding $Q_4$ we see that $v$ has at most six
5-neighbors. This implies that $v$ sends a total charge of at most
$6 \times 1/2 = 3$ to its neighbors, thus $c^{*}(v) \geq 0$.

\medskip
\textbf{10-vertices:} A $10$-vertex $v$ has a charge of +4. Let
$v_1, \dots, v_{10}$ be the neighbors of $v$ in the cyclic order
around $v$. If $v_i$ is a bad 4-neighbor of $v$ and
$\deg(v_{i-1})=7$, $\deg(v_{i+1})= 6$, then the absence of $Q_3$
and $Q_9$ implies that $\deg(v_{i+2}) \geq 6$ and $\deg(v_{i-2})
\geq 5$. The absence of $Q_5$ also implies that if $v_{i+3}$ is
another bad 4-neighbor, then $\deg(v_{i+2})=7$, thus
$\deg(v_{i+4})=6$ and $\deg(v_{i+5})\geq 6$ (all indices modulo
10). By excluding $Q_{23}$ and $Q_4$, we conclude that if $v$ has
two bad 4-neighbors, then it has no other 4-neighbor and has at
most two 5-neighbors. This implies that $c^{*}(v) \geq 0$. Suppose
now that $v$ has precisely one bad 4-neighbor, say $v_2$. We may
assume $\deg(v_1) = 7$, $\deg(v_3)=6$ and by the arguments given
above, $\deg(v_{10}) \geq 5$, $\deg(v_4) \geq 6$. Excluding $Q_4$,
$v$ can have at most four 5-neighbors. Thus, the only possibility
that $c^{*}(v) < 0$ is that $v$ has three more 4-neighbors (and
the only way to have this is that the 4-neighbors are $v_5, v_7,
v_9$) or that $v$ has two more 4-neighbors and two 5-neighbors (in
which case 4-neighbors are $v_5$, $v_7$ and 5-neighbors are $v_9$,
$v_{10}$). In each of these cases, we see, by excluding $Q_3$ and
$Q_5$, that $\deg(v_4) \geq 7$, $\deg(v_6) \geq 7$ and $\deg(v_8)
\geq 7$. Thus, excluding $Q_9$, $v$ sends charge of at most $3/4$
to each of $v_5$ and $v_7$ and at most 1 together to both $v_9$
and $v_{10}$. Hence, $c^{*}(v) \geq 4 - 3/2 - 2 \times 3/4 - 1 =
0$.

Suppose now that $v$ has no bad 4-neighbors. If $v$ has five
4-neighbors, then they are (without loss of generality) $v_1, v_3,
v_5, v_7, v_9$, and excluding $Q_3$ we see that $\deg(v_j) \geq 7$
for $j= 2, 4, 6, 8, 10$. This implies (by the argument as used
above) that $v$ sends charge of at most $3/4$ to each 4-neighbor,
thus $c^{*}(v) \geq 4 - 5 \times 3/4 > 0$. Similarly, if $v$ has
one 5-neighbor $v_1$ and four 4-neighbors $v_3, v_5, v_7, v_9$,
then we see as above that $v$ sends charge of at most $3/4$ to
each 4-neighbor, and thus $c^{*}(v) \geq 4 - 4 \times 3/4 - 1/2 >
0$. If $v$ has three 4-neighbors, then the exclusion of $Q_4$
implies that it has at most two 5-neighbors. Similarly, if $v$ has
two 4-neighbors, then it has at most four 5-neighbors. If $v$ has
one 4-neighbor, then it has at most five 5-neighbors. If $v$ has
no 4-neighbors, it has at most six 5-neighbors. In each case,
$c^{*}(v) \geq 0$.

\medskip
\textbf{$11^{+}$-vertices:} Let $v$ be a $d$-vertex, with $d \geq
11$. Let $v_1,\dots,v_d$ be the neighbors of $v$ in cyclic
clockwise order, indices modulo $d$. Suppose that $v_i$ is a bad
4-vertex. Then we may assume that $\deg(v_{i-1})=7$ and
$\deg(v_{i+1})=6$ (or vice versa), since $Q_7$ is excluded. By
noting that the fourth neighbor of $v_i$ has degree 6, we see that
$\deg(v_{i+2}) \geq 6$ (since $Q_3$ is excluded) and
$\deg(v_{i-2}) \geq 5$ (since $Q_{9}$ is excluded). If $v_i$ is a
good 4-vertex, then its neighbors are $6^{+}$-vertices. Now, we
redistribute the charge sent from $v$ to its neighbors so that
from each bad 4-vertex $v_i$ we give $1/2$ to $v_{i-1}$ and $1/2$
to $v_{i+1}$, and from each good 4-vertex $v_i$ we give $1/4$ to
$v_{i-1}$ and $1/4$ to $v_{i+1}$. We claim that after the
redistribution, each neighbor of $v$ receives from $v$ at most
$1/2$ charge in total. This is clear for 4-neighbors of $v$. A
5-neighbor of $v$ is not adjacent to a 4-vertex, so it gets charge
of at most $1/2$ as well. The claim is clear for each 6-neighbor
of $v$ since it is adjacent to at most one 4-vertex ($Q_3$ is
excluded). If a 7-neighbor $v_j$ of $v$ satisfies $\deg(v_{j+1}) =
\deg(v_{j-1}) = 4$, the exclusion of $Q_9$ implies that both
$v_{j-1}$ and $v_{j+1}$ are good 4-vertices. Thus, the claim holds
for 7-neighbors of $v$. An $8^{+}$-neighbor of $v$ cannot be
adjacent to a bad 4-neighbor of $v$, and therefore it receives
charge of at most $1/2$ from $v$ after the redistribution. This
implies that if $d \geq 12$, then the final charge at $v$ is
$c^{*}(v)\geq c(v)-\tfrac{1}{2}d \geq 0$.

It remains to consider the case when $d=11$. In this case the same
conclusion as above can be made if we show that either the
redistributed charge at one of the vertices $v_i$ is 0, or that
there are two vertices whose redistributed charge is at most
$1/4$. If there exists a good 4-vertex, then there exists a good
4-vertex $v_i$, one of whose neighbors, say $v_{i-1}$, gets $1/4$
total redistributed charge. This is easy to see since $d=11$ is
odd and $Q_3$ and $Q_{9}$ are excluded. Let $t \geq 0$ be the
largest integer such that $v_i,v_{i+2},\dots,v_{i+2t}$ are all
good 4-neighbors of $v$. Then it is clear that $v_{i+2t+1}$ has
total redistributed charge $1/4$ and that $v_{i-1}\ne v_{i+2t+1}$
(by parity). This shows that the total charge sent from $v$ is at
most $5$, thus the final charge $c^{*}(v)$ is non-negative. Thus,
we may assume that $v$ has no good 4-neighbors. If $v$ has a bad
4-neighbor $v_i$, then we may assume that $\deg(v_{i-1})=7$ and
$\deg(v_{i+1})=6$. As mentioned above, we conclude that
$\deg(v_{i+2})\ge6$. We are done if this vertex has 0
redistributed charge. Otherwise, $v_{i+2}$ is adjacent to another
bad 4-neighbor $v_{i+3}$ of $v$. Since
$v_i,v_{i+1},v_{i+2},v_{i+3}$ do not correspond to the excluded
configuration $Q_5$, we conclude that $\deg(v_{i+2})=7$. Now we
can repeat the argument with $v_{i+3}$ to conclude that $v_{i+6},
v_{i+9}$ are also bad 4-vertices and $\deg(v_{i+8})=7$. However,
since $\deg(v_{i-1})=7$, we conclude that $v_{i+9}$ cannot be a
bad 4-vertex and hence there is a neighbor of $v$ with
redistributed charge 0.

Thus, $v$ has no 4-neighbors. Now the only way to send charge
$1/2$ to each neighbor of $v$ is that all neighbors of $v$ are
5-vertices. However, in this case we have the configuration $Q_4$.

To summarize, we have shown that the final charge of each vertex
is non-negative and this completes the proof.
\end{proof}

\section{Reducibility}

This section is devoted to the reducibility part of the proof of
our main result (Theorem \ref{thm: main 3-list}) using the
unavoidable configurations in Fig.~1. Let $G$ be a (projective)
planar graph and $L$ a $3$-list-assignment. It is sufficient to
prove the theorem when $G$ is a triangulation. Otherwise, we
triangulate $G$ and any $L$-coloring of the triangulation is an
$L$-coloring of $G$.\footnote[1]{While this argument is standard
for planar graphs, it is much less clear (and only conditionally
true) for the case of projective plane. The details about this
case are provided in the next section.} Of course, we only
consider arboreal $L$-colorings, and we omit the adverb
``arboreal" in the sequel.

A configuration $C$ contained in $G$ is called \DEF{reducible} if
$|C| \leq 9$ and any $L$-coloring of $G - V(C)$ can be extended to
an $L$-coloring of $G$ in at least two ways. Showing that every
triangulation $G$ contains a reducible configuration will imply
that $G$ has at least $2^{|V(G)|/9}$ arboreal $L$-colorings.

Here we prove our main theorem by showing that each configuration
from Section 5.2 is reducible. The following lemma will be used
throughout this section to prove reducibility.

\begin{lemma}
\label{lem:Basic Lemma} Let $G$~be a planar graph, $L$ a
$3$-list-assignment for $G$, and $v_1,\dots,v_k \in V(G)$. Let \/
$G_i = G - \{v_{i+1},\dots, v_k \}$ for $i=0,\dots,k$ and consider
the following properties:
\begin{enumerate}
\item[\rm (1)] For every $i=1,\dots,k$, $deg_{G_i}(v_i) \leq 5$.
\item[\rm (2)] There exists an $i$ such that $deg_{G_i}(v_i) \leq
3$.
\end{enumerate}
If $(1)$ holds, then every arboreal $L$-coloring of $G_0$ can be
extended to $G$. If both $(1)$ and $(2)$ hold, then every arboreal
$L$-coloring of $G_0$ can be extended to $G$ in at least two ways.
\end{lemma}

\begin{proof}
Let $f$ be an $L$-coloring of $G_0$. Since $v_1$ has degree at
most 5 in $G_1$, there is a color $c \in L(v_1)$ such that $c$
appears at most once on $N_{G_1}(v_1)$. Therefore, coloring $v_1$
with $c$ gives an $L$-coloring of $G_1$. Repeating this argument,
we see that the $L$-coloring of $G_0$ can be extended to an
$L$-coloring of $G$ by consecutively $L$-coloring $v_1, v_2,\dots,
v_k$. If (2) holds for $i$, then there are actually two possible
colors that can be used to color $v_i$. Therefore, every
$L$-coloring of $G_0$ can be extended to $G$ in at least two ways.
\end{proof}

\begin{lemma}
\label{lem:C1} Configurations $Q_1,\dots, Q_5$, $Q_8,\dots,
Q_{13}$, $Q_{15},\dots, Q_{22}$ listed in Fig.~\ref{fig:1} are
reducible. The configuration $Q_{23}'$ that is obtained from
$Q_{23}$ by deleting the pendant vertex with $\delta(v)=4$ is also
reducible.
\end{lemma}

\begin{proof}
For these configurations $Q_i$ and $Q'_{23}$ we simply apply Lemma
\ref{lem:Basic Lemma}. The corresponding enumeration $v_1,\dots,
v_k$ ($k=|V(Q_i)|$ or $k=|V(Q'_{23})|$) is shown in Figure
\ref{fig:2}. The vertex for which condition (2) of Lemma
\ref{lem:Basic Lemma} applies is always $v_1$; it is shown by a
larger circle.
\end{proof}

\begin{figure}[t!]
   \centering
   \includegraphics[width=11cm]{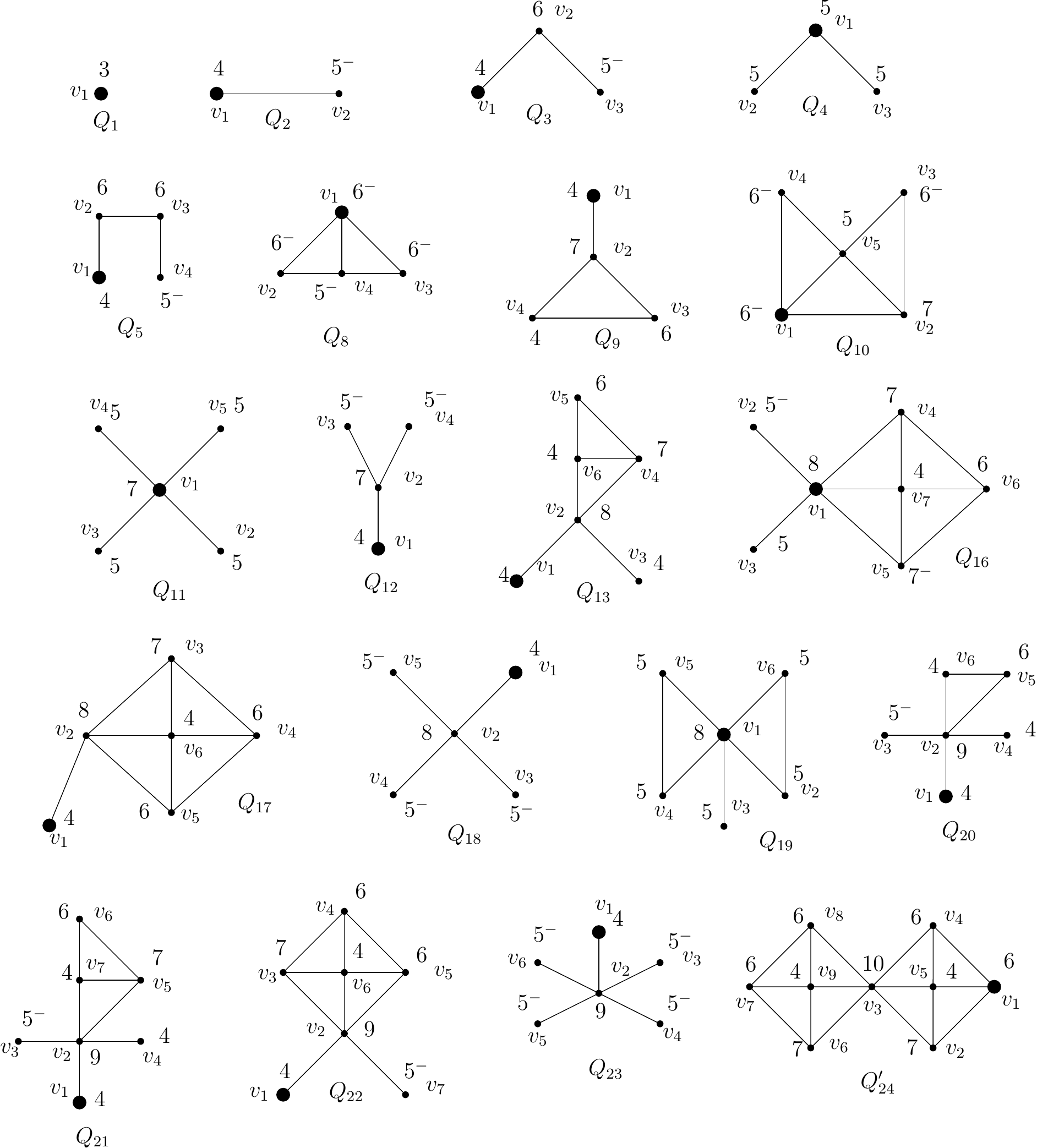}
   \caption{Lemma \ref{lem:Basic Lemma} applies to several configurations.}
   \label{fig:2}
\end{figure}

\begin{lemma}
Configuration $Q_6$ in Fig.~\ref{fig:1} is reducible.
\end{lemma}

\begin{proof}
Let $u$ be the 4-vertex and let $u_1, u_2, u_3, u_4$ be its
neighbors in cyclic order and let $C$ be the cycle $u_1u_2u_3u_4$.
Suppose that $\deg(u_1) = \deg(u_2) = 7$, $\deg(u_3) \le 7$ and
$\deg(u_4) = 6$. Let $f$ be an $L$-coloring of \/ $G - \{u, u_1,
u_2, u_3, u_4 \}$. Now, consider $u_2$. If there are at least two
ways to extend the coloring $f$ to $u_2$, then we can obtain at
least two different colorings for $G$ by sequentially coloring
$u_1, u_3, u_4, u$ using Lemma \ref{lem:Basic Lemma}. Therefore,
we may assume that $L(u_2) = \{1,2,3\}$ and that colors 1 and 2
each appear exactly twice on $N(u_2)$. Now, let us color $u_2$
with color $3$. We now consider coloring $u_1$ and $u_3$. We claim
that at least one of $u_1$ and $u_3$ must be forced to be colored
3. Otherwise, we color $u_1$ and $u_3$ without using color 3, then
we color $u_4$ arbitrarily (this is possible since $u$ is yet
uncolored). Now, if $3 \in L(u)$, then we can color $u$ with $3$
since $u_2$ has no neighbor of color 3 and hence it is not
possible to make a cycle colored 3. Moreover, there is at most one
color (other than color 3) that can appear on the neighborhood of
$u$ twice. Therefore, $u$ has another available color in its list
and so there are two ways to color $u$. Similarly, we get two
different colorings of $u$ when $3 \notin L(u)$. This proves the
claim, and we may assume that $L(u_1) = \{a, b, 3\}$, $u_1$ is
forced to be colored 3, and that the four colored neighbors of
$u_1$ not on $C$ have colors $a, a, b, b$. Now, we color $u_3$
arbitrarily with a color $c$. We may assume that $c \neq 3$, for
otherwise we color $u_4$ arbitrarily and we will have two
available colors for $u$. To complete the proof it is sufficient
to show that $u_4$ can be colored with a color that is not $c$,
for then we could color $u$ with at least two different colors. If
$u_4$ is forced to be colored $c$, then for every color $x \in
L(u_4)$, $x \neq c$, the color $x$ must appear at least twice on
$N(u_4)$. This implies that the three colored neighbors of $u_4$
not on the cycle have colors $3, y, y$, for some color $y$ and
that $3,y \in L(u_4)$. But recall that $u_1$ and $u_2$ have no
neighbors outside $C$ having color $3$. Therefore, coloring $u_4$
with color 3 gives a proper coloring of $G-u$. Now, $u$ can be
colored with at least two colors to obtain a coloring of $G$.
\end{proof}

\begin{lemma}
\label{lem:4 8-6-7-6} \label{lem:C2} Let $u$ be a 4-vertex, and
suppose $u_1, u_2, u_3, u_4$ are the neighbors of $u$ in cyclic
order. Suppose that $\deg(u_1) \leq 6$, $\deg(u_2) \leq 7$ and
$\deg(u_3) \leq 6$. This configuration is reducible. In
particular, the configuration $Q_7$ in Fig.~\ref{fig:1} is
reducible.
\end{lemma}

\begin{proof}
Let $f$ be an $L$-coloring of $G'=G- \{u, u_1, u_2, u_3\}$.
Suppose that $f(u_4) = 3$. Now, consider $u_1$. Note that we can
extend the coloring of $G'$ to $u_1, u_2, u_3, u$ (in this order)
by Lemma \ref{lem:Basic Lemma}. Suppose, for a contradiction, that
$f$ has only one extension to an $L$-coloring of $G$. Then colors
of each of $u_1, u_2, u_3, u$ are uniquely determined in each step
and two colors from each vertex list are forbidden. Now, consider
$u_1$. Since only four of its neighbors are colored and $f(u_4)
=3$, we can color $u_1$ with a color other than 3, say 2, and we
may further assume that its colored neighbors use colors 1 and 3
twice, where $L(u_1) = \{1, 2, 3\}$. Now, consider coloring $u_2$.
The color 2 at $u_1$ cannot create a monochromatic cycle
containing $u_2$. Thus, the only way for a color of $u_2$ to be
forced is that $L(u_2) = \{a, b, x\}$ and colors $a$ and $b$ each
appear twice on $N(u_2) \backslash \{u_1\}$. In this case, we
color $u_2$ with the color $x$. Similarly, $x$ does not give any
restriction for a color at $u_3$, so $u_3$ satisfies $L(u_3) =
\{3, c, y\}$ and the three neighbors of $u_3$ distinct from $u_4$
are colored with colors $3, c, c$. Now, if $u$ does not have two
colors on $N(u)$, each appearing twice, we have two different
available colors in $L(u)$. Therefore, we may assume that $\{x,
y\}= \{2, 3\}$ and that $2,3 \in L(u)$. Since $L(u_3) = \{3, c,
y\}$, it follows that $y = 2$ and $x=3$. Now we see that coloring
$u$ with color 3 does not create a monochromatic cycle, so $u$ has
two available colors: color 3 and $z \in L(u) \backslash \{2,3\}$.
%
%
\end{proof}

\begin{lemma}
\label{lem:8 4-4-5 adj. 6} \label{lem:C4} 
The configuration $Q_{14}$ is reducible.
\end{lemma}

\begin{proof}

Let $u$ be an 8-vertex and assume its neighbors (in the clockwise
cyclic order) are $u_1,\dots,u_8$ and let $C$ be the 8-cycle
$u_1u_2\dots u_8u_1$. Suppose that $\deg(u_i) = \deg(u_j) = 4$,
$\deg(u_k) \leq 5$ and $\deg(u_l)=6$, where $i,j,k,l \in
\{1,\dots, 8\}$ and $i \neq j$. Assume that $u_l$ and $u_j$ are
adjacent on $C$. We may assume that $u_l = u_{j+1}$. If $u_i =
u_{j+2}$, then we can use Lemma \ref{lem:Basic Lemma} (with $v_1 =
u_i, v_2 = u, v_3 = u_{j+1}, v_4 = u_j, v_5 = u_k$), where
property (2) applies for $v_1$.

Therefore, we may assume that $u_i \neq u_{j+2}$. Let $L(u) =
\{1,2,3\}$ and consider an $L$-coloring $f$ of $G - \{u, u_i, u_j,
u_k, u_l\}$. Without loss of generality, we may assume that colors
1 and 2 each appear exactly twice on $N(u)$ in the coloring $f$.
Otherwise, there are two ways to extend the coloring $f$ of $G -
\{u, u_i, u_j, u_k, u_l\}$ to a coloring of $G - \{u_i, u_j, u_k,
u_l\}$, and applying Lemma \ref{lem:Basic Lemma} we can extend
each of these to a coloring of $G$. Therefore, color 3 does not
appear in the neighborhood of $u$ in the coloring $f$. We color
$u$ with color $3$ to obtain a coloring $g$ of $G - \{u_i, u_j,
u_k, u_l\}$. Now, consider the 6-vertex $u_{j+1}$. Since $u_{j+1}$
has at most five colored neighbors so far, we have at least one
available color for it from its list. If $3 \notin L(u_{j+1})$ we
color $u_{j+1}$ arbitrarily with an available color. If $3 \in
L(u_{j+1})$, we color $u_{j+1}$ with $3$ if color 3 does not
appear on $N(u_{j+1})\backslash \{u\}$. If color 3 appears on
$N(u_{j+1})\backslash \{u\}$, we color $u_{j+1}$ with any other
available color from its list except 3 (this is possible since the
remaining three colored neighbors of $u_{j+1}$ can forbid only one
additional color from $L(u_{j+1})$). Now, consider $u_i$. We know
that $u_i \neq u_{j+2}$. First, assume that $3 \notin L(u_i)$.
Since $u_i$ has only three colored neighbors and $u$ is colored 3,
there are at least two available colors in $L(u_i)$ that can be
used to color $u_i$. Each coloring then can be extended to a
coloring of $G$ by Lemma \ref{lem:Basic Lemma}. Therefore, we may
assume that $3 \in L(u_i)$. Recall that no neighbor of $u$, except
possibly $u_{j+1}$, is colored 3, and if so, then $u_{j+1}$ has no
neighbor besides $u$ of color 3. Therefore, $u_i$ can be colored
with color 3 without creating a monochromatic cycle of color 3.
Consequently, the four colored neighbors of $u_i$ can forbid at
most one color from $L(u_i)$, which implies that we can color
$u_i$ with two different colors. Now, applying Lemma
\ref{lem:Basic Lemma} to $G-\{u_k,u_j\}$, we see that each of
these two colorings can be extended to a coloring of $G$.
\end{proof}

\section{Proof of the main theorem}

It is easy to see that every plane graph is a spanning subgraph of
a triangulation; we can always add edges joining distinct
nonadjacent vertices until we obtain a triangulation. However,
graphs in the projective plane no longer satisfy this property.
The following extension will be sufficient for our purpose.

\begin{proposition}
\label{prop:pp} Let $G$ be a graph embeddable in the projective
plane. Then one of the following holds:
\begin{itemize}
\item[\rm (a)] $G$ is a spanning subgraph of a triangulation of
the plane or the projective plane.
\item[\rm (b)] $G$ contains
vertices $u,v$ of degree at most\/ $3$ such that the graph $G-u-v$
is planar.
\item[\rm (c)] $G$ contains adjacent vertices $u,v$ of
degree at most\/ $4$ such that the graph $G-u-v$ is planar.
\end{itemize}
\end{proposition}

\begin{proof}
If $G$ is a planar graph, then we have (a); so we may assume that
$G$ is not planar. The proof proceeds by induction on the number
$k=3|V(G)|-|E(G)|-3$. If $k=0$, then $G$ triangulates the
projective plane (cf.\ \cite[Proposition 4.4.4]{MT2001}), and we
have (a). If $G$ is not 2-connected, then we can add an edge
joining two vertices in distinct blocks of $G$ and keep the
embeddability in the projective plane, and we win by induction.
Thus we may assume that $G$ is 2-connected and non-planar. This
assures that facial walks of every embedding of $G$ are cycles of
$G$ (cf.\ \cite[Proposition 5.5.11]{MT2001}). If $G$ is not a
triangulation, then there is a facial cycle $C=v_1v_2\dots
v_rv_1$, where $r\ge 4$. If two vertices of $C$ are nonadjacent in
$G$, we can add the edge joining them and win by induction. Thus,
the subgraph $K$ of $G$ induced on $V(C)$ is the complete graph of
order $r$. Since this subgraph has a facial walk of length $r>3$,
we conclude that $r\in\{4,5\}$ and the induced embedding of $K$ is
as shown in Figure~\ref{fig:3}.

\begin{figure}[htb]
   \centering
   \includegraphics[height=4cm]{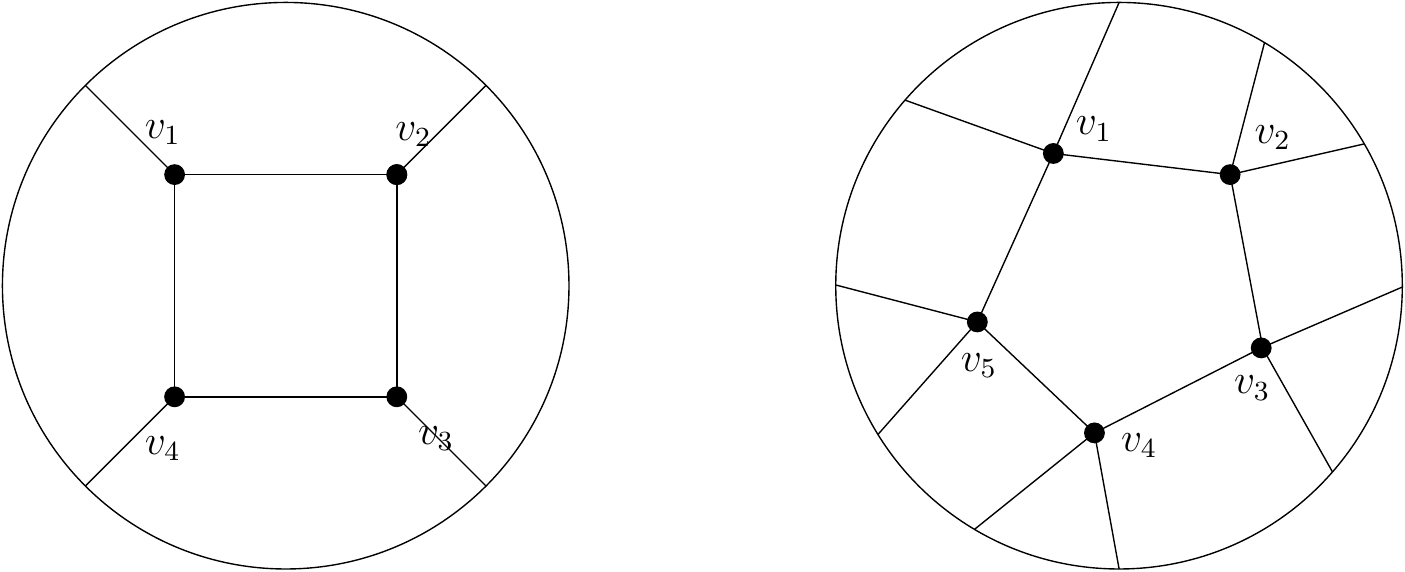}
   \caption{$K_4$ and $K_5$ embedded in the projective plane}
   \label{fig:3}
\end{figure}

Let us consider the vertex $v_1$ and the edges $v_1v_3$ and $v_2v_4$
(if $r=4$), and $v_1v_3$, $v_1v_4$ and $v_2v_5$ (if $r=5$).
These edges are
embedded as shown in Figure \ref{fig:3}. Suppose that $v_1$ has
two neighbors $a,b\notin V(C)$ such that the cyclic order around
$v_1$ is $v_1v_4,v_1a,v_1v_3,v_1b$ when $r=4$ and
$v_1v_5,v_1a,v_1v_s,v_1b$ (where $s=3$ or $s=4$) when $r=5$. Then
we can re-embed the edge $v_1v_3$ (if $r=4$) or re-embed the edges
$v_1v_3$ and $v_1v_4$ (if $r=5$) into the face bounded by $C$ and
then add an edge joining two nonadjacent neighbors of $v_1$.
Again, we are done by applying the induction hypothesis.

Thus we may assume henceforth that all neighbors of each vertex
$v_i$ that are not on $C$ are contained in a single face of $K$.
If a face $F$ of $K$ contains at least one vertex that is not on
$C$, then each vertex of $C$ on the boundary of $F$ has a neighbor
inside $F$. If not, we would be able to add an edge and would be
done by applying induction. Since any two faces of $K$ have a
vertex in common, the aforementioned property implies that at most
one face of $K$ contains any vertices of $G$. If $r=5$, this
implies that (c) is satisfied. Thus $r=4$ and since $G$ is
non-planar, there is a face $F$ of $K$ that contains vertices of
$G$ in its interior. We may assume that $F$ contains the edge
$v_1v_2$ on its boundary. Now, if we re-embed the edge $v_1v_2$
into the face of $K$ distinct from $F$ and $C$, we obtain a new face
containing the former face bounded by $C$ that
is of length at least 5. Thus we get into one of the above cases,
and we are done.
\end{proof}

\begin{proof}[Proof of Theorem \ref{thm: main 3-list}]
The proof is by induction on the number of vertices, $n=|G|$. Let
$L$ be a 3-list-assignment for $G$. Let us first suppose that $G$
is a triangulation. By Theorem \ref{thm:unavoidable} and Lemmas
\ref{lem:C1}--\ref{lem:C4}, $G$ contains a reducible configuration
$C$ on $k\le9$ vertices. By the induction hypothesis, $G-V(C)$ has
at least $2^{(n-k)/9}$ arboreal $L$-colorings. Since $C$ is
reducible, each of these colorings extends to $G$ in at least two
ways, giving at least $2\times 2^{(n-k)/9}\ge 2^{n/9}$ arboreal
$L$-colorings in total.

If $G$ is a spanning subgraph of a triangulation, we apply the above to
the triangulation containing $G$. Otherwise, Proposition \ref{prop:pp}
shows that $G$ contains vertices $u,v$ of low degree such that
$G-u-v$ is a spanning subgraph of a triangulation $G'$.
By the induction hypothesis, $G'$ has at least $2^{(n-2)/9}$
$L$-colorings. By properties (b) and (c) of the proposition,
each of them can be extended to $G$ in at least two ways by applying
Lemma \ref{lem:Basic Lemma}, and we conclude as before.
\end{proof}


\begin{thebibliography}{9}


\bibitem{CKW1968}
G. Chartrand, H.V. Kronk, C.E. Wall,
The point-arboricity of a graph, Israel J. Math. 6 (1968), 169--175.

\bibitem{BFJKM2004}
D. Bokal, G. Fijav\v{z}, M. Juvan, P. M. Kayll, B. Mohar,
The circular chromatic number of a digraph,
J. Graph Theory 46 (2004), 227--240.

%
%

\bibitem{HM2011}
A. Harutyunyan, B. Mohar, Gallai's Theorem for List Coloring of
Digraphs , SIAM Journal on Discrete Mathematics 25 (2011),
170--180.
%
%
%
%
%

\bibitem{KM1974}
H. V. Kronk, J. Mitchem, Critical point arboritic graphs,
J. London Math. Soc. 9 (1974/75), 459--466.


%

\bibitem{MT2001}
B. Mohar, C. Thomassen, Graphs on Surfaces, Johns Hopkins
University Press, 2001.

\bibitem{N1982}
V. Neumann-Lara, The dichromatic number of a digraph,
J.~Combin.\ Theory, Ser. B 33 (1982), 265--270.

\bibitem{N} V. Neumann-Lara, seminar notes (communicated by A.
Bondy and S. Thomass\'{e}).


\bibitem{S2002} R. \v{S}krekovski, On the critical point-arboricity
graphs, J.~Graph Theory 39 (2002), 50--61.
%
%

\end{thebibliography}
\end{document}